\documentclass[12pt]{article}
\usepackage{latexsym, amsmath, amsfonts, amsthm, amssymb}
\usepackage{times}
\usepackage{a4wide}
\usepackage{graphicx, tocloft}
\usepackage{titlesec, titling}
\usepackage{xcolor}
\titleformat{\section}{\large\bfseries}{\thesection}{1em}{}
\titleformat{\subsection}{\normalsize\bfseries}{\thesubsection}{1em}{}

\def \dist {{\rm dist }}

\def \RR {\mathbb R}

\def \NN {\mathbb N}
\def \EE {\mathbb E}
\def \ZZ {\mathbb Z}

\def \TT {\mathbb T}
\def \eps {\varepsilon}

\def \essosc {\mathrm {essOsc}}
\def \osc {\mathrm{Osc}}

\newtheorem{theorem}{Theorem}
\newtheorem{lemma}[theorem]{Lemma}
\newtheorem{proposition}[theorem]{Proposition}

\def\myffrac#1#2 in #3{\raise 2.6pt\hbox{$#3 #1$}\mkern-1.5mu\raise 0.8pt\hbox{$
#3/$}\mkern-1.1mu\lower 1.5pt\hbox{$#3 #2$}}

\def\qed{\hfill $\vcenter{\hrule height .3mm
\hbox {\vrule width .3mm height 2.1mm \kern 2mm \vrule width .3mm
height 2.1mm} \hrule height .3mm}$ \bigskip}

\begin{document}

\title{Lipschitz functions on the infinite-dimensional torus}
\author{Dmitry Faifman\textsuperscript{1}  $\ $ and $\ $ Bo'az Klartag\textsuperscript{2}}
\date{}
\maketitle

\abstract{We discuss the spectrum phenomenon for  Lipschitz functions on the infinite-dimensional torus.
Suppose that $f$ is a measurable, real-valued, Lipschitz function on the torus $\mathbb{T}^{\infty}$.
We prove that there exists a number $a \in \RR$ with the following property: For any $\eps > 0$ there exists a parallel, infinite-dimensional subtorus
$M \subseteq \TT^{\infty}$ such that the restriction of the function $f-a$ to the subtorus $M$ has an $L^{\infty}(M)$-norm of at most $\eps$.}

\footnotetext[1]{Institut des Hautes \'Etudes Scientifiques (IHES), 35 Route de Chartres, 91440 Bures-sur-Yvette, France. Email: dfaifman@gmail.com}
\footnotetext[2]{School of Mathematical Sciences, Tel-Aviv
University, Tel Aviv 69978, Israel. Email: klartagb@tau.ac.il}

\section{Background and Results}

One of the most remarkable phenomena in high dimensions is the emergence of a {\it spectrum} for uniformly continuous functions.
It was shown by Milman in his proof of Dvoretzky's theorem \cite{mil_30} that given any $1$-Lipschitz function $f$ on the high-dimensional
sphere $S^n$, one may find a section of $S^n$ by a linear subspace of large dimension, on which $f$ is nearly a constant function. The value of this constant may be thought of, approximately, as an element in a spectrum associated with  $f$. An analogous effect in discrete mathematics is Ramsey's theorem \cite{alon},
 according to which any coloring of a large complete graph by a fixed number of colors contains a large induced subgraph which is monochromatic.

 \medskip  There have been several attempts to formulate infinite-dimensional analogs of the  Ramsey-–Dvoretzky–-Milman phenomenon.
 Let $X$ be an infinite-dimensional Banach space whose unit sphere is denoted by $S(X)$. For a function $f: S(X) \rightarrow \RR$ one defines its infinite-dimensional spectrum $\sigma(f)$
  as the collection of all values $a \in \RR$ with the following property: For any $\eps > 0$, there exists an infinite-dimensional subspace $Y \subseteq X$ such that
 $$ |f(v) - a| < \eps \qquad \qquad \qquad \text{for all} \ v \in S(Y), $$
 where $S(Y)$ is the unit sphere in the subspace $Y$. A question that was open for many years was
  whether
 the infinite-dimensional spectrum of any Lipschitz function is non-empty.  Unfortunately, even when $X$ is a Hilbert space, the answer is decisively
 negative as  was proven by Odell and Schlumprecht \cite{OS1, OS2}. On the positive side,  Gowers \cite{gowers} proved that any Lipschitz function on the Banach space $c_0$ admits a non-empty infinite-dimensional spectrum. The space $c_0$ is essentially the only separable Banach space for which the answer is positive, as was proven in \cite{OS1, OS2}.

 \medskip Here we investigate the question of existence of an infinite-dimensional spectrum in a different situation,
 that of the infinite-dimensional torus, or of infinite-dimensional product spaces in general. Lipschitz functions on a finite-dimensional torus were analyzed
  using probabilistic tools by Faifman, Klartag and Milman  \cite{FKM}. In this paper  we will exploit
 the fact that the infinite-dimensional torus admits a product probability measure,
which allows one to use probabilistic arguments akin to the finite-dimensional case.

\medskip {Let us introduce some terminology and notation and recall a few basic facts that are used throughout the paper.} The infinite-dimensional torus is typically denoted by $\TT^{\NN}$ or by $\TT^{\infty}$. An element $x \in \TT^{\infty}$
is a sequence $x = (x_i)_{i=1,2, \ldots}$ with $x_i \in \TT = \RR / \ZZ$ for all $i$.
Write $\sigma$ for the uniform probability measure on $\TT^{\infty}$, which is a complete product measure,
invariant under translations.
When we say that a random point $X$ is distributed uniformly on $\TT^{\infty}$
or when we say that a function $f$ on $\TT^{\infty}$ is measurable, we always refer to the probability measure $\sigma$.
For $x,y \in \TT^{\infty}$ consider the Euclidean metric
$$  \dist(x,y) = \sqrt{ \sum_{i=1}^{\infty} \dist^2(x_i, y_i) } $$
where $ x = (x_i)_{i \geq 1}, y = (y_i)_{i \geq 1}$ and where $\dist(x_i, y_i)$ is the distance between $x_i$ and $y_i$ in the circle $\RR / \ZZ$.
It may happen that $d(x,y) = +\infty$ for some $x,y \in \TT^{\infty}$. In fact, the torus $\TT^{\infty}$ is split into infinitely-many connected
components with respect to the metric $\dist$, all of  measure zero.
It is explained in Gromov \cite[Section 3.1]{gromov} that
for any measurable subsets $A, B \subseteq \TT^{\infty}$ with $\sigma(A) \cdot \sigma(B) > 0$,
$$ \inf_{x \in a, y \in B} \dist(x,y) \leq C( \sigma(A), \sigma(B) ) < \infty $$
for a certain explicit function $C: (0,1] \times (0,1] \rightarrow [0, \infty)$.
A subset  $M \subseteq \TT^{\infty}$ is
a {\it parallel infinite-dimensional subtorus} if there exist an infinite subset $A \subseteq \NN$ and values $b: \NN \setminus A \rightarrow \TT$ such that
$$ M = \left \{  (x_i)_{i \geq 1} \in \TT^{\infty} \, ; \, x_i = b_i \ \text{for all} \ i \in \NN \setminus A \right \}. $$
Note that the uniform probability measure on the infinite-dimensional subtorus $M$ is well-defined, thus one may speak of  the space $L^{\infty}(M)$. Our main result is the following:

\begin{theorem}
For any measurable function $f:\mathbb{T}^{\infty} \rightarrow \mathbb{R}$
that is Lipschitz with respect to the Euclidean metric $\dist$, there exists $a \in \RR$ with the following property:
For any $\eps > 0$ there exists a parallel infinite-dimensional subtorus
$M \subseteq \TT^{\infty}$ such that
$ \left \| f - a \right \|_{L^{\infty}(M)} < \eps$.
\label{thm1}
\end{theorem}

Theorem \ref{thm1} thus implies that any measurable, Lipschitz function on $\TT^{\infty}$ has a non-empty spectrum in an appropriate sense.
In order to have in mind some concrete examples
of measurable functions on $\TT^{\infty}$, we mention the function
\begin{equation} f(x)=\sum_{i=1}^\infty a_i \cos(2\pi x_i) \qquad \qquad \qquad (x \in \TT^{\infty}) \label{eq_409} \end{equation}
where $\cos(2 \pi x_i)$ is clearly well-defined for $x_i \in \TT = \RR / \ZZ$. By Kolmogorov's three-series theorem (see, e.g., Kahane \cite[Section 3]{kahane}),
 the series in (\ref{eq_409}) converges almost everywhere if and only if $\sum a_i^2<\infty$.
Assuming that indeed $\sum a_i^2 < \infty$, the function $f$ is a well-defined\footnote{Formally, $f$ is well-defined only almost everywhere with respect to $\sigma$.  For completeness, let us agree that
$f$ attains the value zero at the few points $x$ for which the series in (\ref{eq_409}) diverges.}, measurable
function on $\TT^{\infty}$ which is in fact Lipschitz with respect to
the Euclidean metric $\dist$.

\medskip We proceed to discuss the necessity of the assumptions of Theorem \ref{thm1}.
The condition of measurability is essential: Indeed, fixing a representative $x^C$ for each $\dist$-connected component $C$ of $\mathbb T^\infty$, and letting $f(x)=\inf_C\dist (x,x^C)$, we get a Lipschitz function, the restriction of which to every parallel, infinite-dimensional subtorus has arbitrarily large values.
An example of a measurable function which is non-Lipschitz and has an empty spectrum may be constructed as follows:
It is well-known that there exists a Borel subset $B \subset \RR$ such that both $B$ and $\RR \setminus B$ intersect any non-empty interval in a set
of a positive Lebesgue measure. Consider the set
$$ A = \left \{ (x_1,x_2,\ldots) \in \TT^{\infty} \, ; \, \sum_{i=1}^{\infty} \frac{
\cos(2 \pi x_i)}{i^2} \in B \right \}. $$
Then the indicator function $f = 1_A$ is a measurable function which has no spectrum.

\medskip In general, a measurable, $\dist$-Lipschitz function need not be continuous with respect to the usual product topology on $\TT^{\infty}$.
The function in (\ref{eq_409}) is continuous with respect to the product topology only under the stronger requirement that $\sum |a_i| < \infty$.
  For a function $f:\mathbb T^\infty \to\mathbb R$ that is continuous in the product topology, its image coincides with its spectrum.
  This is because every element of the basis of the topology contains a parallel infinite-dimensional subtorus of the form $M = \{ (x_i)_{i \geq 1} \in \TT^{\infty} \, ; \,
  x_i = b_i \ \forall i < N \}$.

\medskip In addition to the Euclidean metric $\dist$, one defines for $1 \leq p \leq \infty$ and $x, y \in \TT^{\infty}$ the distance $\dist_p$ by
\begin{equation} \dist_{p}(x,y) = \left( \sum_{i=1}^{\infty} \dist^p(x_i, y_i) \right)^{1/p}, \qquad \qquad \qquad \label{eq_450} \end{equation}
where the case $p = \infty$ is defined by $\dist_{\infty}(x,y) = \sup_{i\geq 1} \dist(x_i, y_i)$.
 Theorem \ref{thm1} is the case $p = 2$ of the following:

\begin{theorem}
For any $1 < p \leq \infty$ and a measurable function $f:\mathbb{T}^{\infty} \rightarrow \mathbb{R}$
which is Lipschitz with respect to the metric $\dist_p$, there exists $a \in \RR$ with the following property:
For any $\eps > 0$ there exists a parallel infinite-dimensional subtorus
$M \subseteq \TT^{\infty}$ such that
$ \left \| f - a \right \|_{L^{\infty}(M)} < \eps$.
\label{thm2}
\end{theorem}

It is only the product structure of $\TT^{\infty}$ that plays a fundamental role in the  proof of Theorem \ref{thm2} given below.
For instance, one may replace
 the infinite product of circles $\TT^{\infty}$ by the infinite-dimensional cube $[0,1]^{\infty}$, or more generally, by an infinite product of the form
 $$ X = X_1 \times X_2 \times \ldots $$
 where $X_1,X_2,\ldots$ are connected Riemannian manifolds with boundary, all of volume one, that have a ``uniformly bounded geometry''. By the last phrase
 we mean that the dimensions, diameters and sectional curvatures of the $X_i$'s should
  all be uniformly bounded. The distance function $\dist_p$ on $X$ is still given by (\ref{eq_450}).
  For concreteness, we provide the statement and proof only for the toric case.
  We believe that the adaptation of our proof of
  Theorem \ref{thm2}
   to the cube $[0,1]^{\infty}$ or to the case of a more general product space is rather straightforward.

\medskip We are not sure whether the conclusion of Theorem \ref{thm2} holds true also for $p=1$.
It could be interesting to investigate whether for $p = \infty$, the essential supremum 
in the conclusion of Theorem \ref{thm2} may be replaced by a supremum.
Let us also comment that the full axiom of choice is not used in the proof of Theorem \ref{thm2},
and that the axiom of dependent choice suffices for our argument.

\medskip {\it Acknowledgements.}
The authors would like to thank Ronen Eldan and Vitali Milman for their interest in this work and for related discussions.
The second named author was supported by a grant from the European Research Council.

\section{Proofs}

Consider the $n$-dimensional torus $\TT^n = \RR^n / \ZZ^n$.
The coordinate vector fields $\frac{\partial}{\partial x_1},\ldots,
 \frac{\partial}{\partial x_n}$ are well-defined on the torus $\TT^n$.
 The metric $\dist_p$  on the finite-dimensional torus $\TT^n$ is defined via
  a formula analogous to (\ref{eq_450}) in which the sum runs only up to $n$.
  For a function $f: \TT^n \rightarrow \RR$ we define its oscillation via
  $$ \osc(f ; \TT^n) = \sup_{\TT^n} f - \inf_{\TT^n} f = \sup_{x,y \in \TT^n} |f(x) - f(y)|. $$
 According to the Rademacher theorem from real analysis,
 any function on $\TT^{n}$ which is Lipschitz with respect to $\dist_p$,  for some $1 \leq p \leq \infty$, is differentiable
 almost-everywhere.
  Let $\omega_{n,p}$ denote the $n$-dimensional volume of the $\ell_p$-ball $B_p^n = \{ x \in \RR^n \, ; \, \sum_{i=1}^n |x_i|^p \leq 1 \}$.
  In this note, all integrals on tori and subtori are
carried out with respect to the the uniform probability measure on the  torus. We will need the following variant of Morrey's inequality:

\begin{lemma}
\label{Coordinate_Maurey} Let $n \geq 1,p \in (1, \infty], 0 < \eps < 1/2$ and let $f:\mathbb{T}^{n}\to\mathbb{R}$ be
$1$-Lipschitz with respect to the metric $\dist_p$. Denote $q = p / (p-1)$, with $q = 1$ in case $p = \infty$. Assume that
\begin{equation}
\int_{\mathbb{T}^{n}}\sum_{i=0}^{n-1}
\frac{2^{i^2+qi}}{\omega_{i,p}\eps^{i+q}}\cdot \Bigg|\frac{\partial f}{\partial x_{i+1}}\Bigg|^{q}\leq1.
\label{eq_950}
\end{equation}
Then $\osc(f;\mathbb{T}^{n})<8\eps$. \end{lemma}
\begin{proof}
Let $\pi: \RR^n \rightarrow \TT^n = \RR^n / \ZZ^n$ be the quotient map.   For a point $x \in \TT^n$ we denote
$$ \ell_i = \pi \left ( \RR e_i \right ), \qquad \qquad E_i = \pi \left ( Sp(e_{1},...,e_{i}) \right ), $$
where $e_1,\ldots,e_n$ are the standard unit vectors in $\RR^n$ and where $Sp(e_{1},...,e_{i})$ is the subspace spanned
by $e_1,\ldots,e_i$. We also denote $x + A = \{ x + y \, ; \, y \in A \}$ for a subset $A \subseteq \TT^n$ and a point $x \in \TT^n$.
 Thus, $x + \ell_i \subseteq \TT^n$ is a one-dimensional torus in $\TT^n$ passing through $x$ in the direction of
$\partial / \partial x_i$. The subtorus $x + E_i \subseteq \TT^n$ is $i$-dimensional, and the vector fields $\partial / \partial x_1,\ldots,
 \partial / \partial x_i$ are tangent to $x + E_i$ at the point $x$.

\medskip  Fix a point $P\in\mathbb{T}^{n}$.
For a decreasing index $i=n,\ldots,0$, we recursively define the random points $P_i, P'_i \in \TT^n$
via the following rules:
\begin{enumerate}
\item[(i)] $\displaystyle P_n = P$.
\item[(ii)] The point $P'_i$ is distributed uniformly in the $i$-dimensional ball $B_p^{i}(P_{i+1},\frac{\eps}{2^{i}})$,
where $B_p^{i}(P_{i+1},\frac{\eps}{2^{i}})$ is the $\dist_p$-ball in the subtorus $P_{i+1} + E_i$
centered at $P_{i+1}$ of radius $\frac{\eps}{2^{i}}$.
\item[(iii)] The point $P_i$ is distributed uniformly in the $1$-dimensional subtorus $P_i' + \ell_{i+1}$.
\end{enumerate}
Note that our recursive definition has a decreasing index, thus we first define $P_n$, then $P'_{n-1}$, then $P_{n-1}$, etc.
Since $f$ is Lipschitz, for $i=0,\ldots,n-1$,
\begin{equation}
\mathbb{E}|f(P_{i}')-f(P_{i})| \leq \mathbb{E} \int_{P_i^{\prime} + \ell_{i+1}} \left| \frac{\partial f}{\partial x_{i+1}} \right|
= \mathbb{E} \left| \frac{\partial f}{\partial x_{i+1}} (P_i) \right|. \label{eq_856}
\end{equation}
By an inductive argument, we see that the last $n-i$
coordinates of the random point $P_i$ are independent random variables that are distributed uniformly over the circle $\TT$.
Let $A_{i+1} \in \TT^i$ be the vector which consists of the first $i$ coordinates of $P_{i+1}$.
We also write $B_p^i(A_{i+1}, r)$ for the $\dist_p$-ball of radius $r$ centered at $A_{i+1}$ in the torus $\TT^i$.
Since $\eps < 1/2$,
\begin{equation}
\mathbb{E} \left| \frac{\partial f}{\partial x_{i+1}} (P_i) \right|
= \mathbb{E} \frac{\int_{ B_p^{i}(A_{i+1},\frac{\eps}{2^{i}}) \times \mathbb{T}^{n-i}} \left|\frac{\partial f}{\partial x_{i+1}} \right|
}{{\rm Vol}_i \left( B_p^{i} \left(A_{i+1},\frac{\eps}{2^{i}} \right) \right)}  =
 \frac{\mathbb{E} \int_{B_p^{i}(A_{i+1},\frac{\eps}{2^{i}}) \times \mathbb{T}^{n-i}} \left|\frac{\partial f}{\partial x_{i+1}} \right|
}{\omega_{i,p}\left(\frac{\eps}{2^{i}}\right)^{i}}. \label{eq_916}
\end{equation}
From (\ref{eq_856}), (\ref{eq_916}) and the H\"older inequality, for $i=0,\ldots,n-1$,
\begin{align} \nonumber
\mathbb{E}|f(P_{i}')-f(P_{i})| & \leq \left(\int_{\mathbb{T}^{n}}\Bigg|\frac{\partial f}{\partial x_{i+1}}\Bigg|^{q}\right)^{\frac{1}{q}}\left(\omega_{i,p}\left(\frac{\eps}{2^{i}}\right)^{i}\right)^{-\frac{1}{q}} \\ &
\leq \left(\frac{2^{i^{2}+qi}}{\omega_{i,p}\eps^{i+q}}\right)^{-\frac{1}{q}}\left(\frac{\omega_{i,p}\eps}{2^{i^{2}}}^{i}\right)^{-\frac{1}{q}}=\frac{\eps}{2^i},
\label{eq_921} \nonumber
\end{align}
where we used our assumption (\ref{eq_950}) in the last passage. The function $f$ is $1$-Lipschitz with respect to $\dist_p$, and hence
$|f(P_{i}')-f(P_{i+1})| \leq \eps / 2^i$ with probability one. Consequently,
$$ \EE |f(P) - f(P_0)| \leq \EE \sum_{i=0}^{n-1}
|f(P_{i})-f(P'_{i})| + |f(P_{i}')-f(P_{i+1})| \leq \sum_{i=0}^{n-1} \frac{2 \eps}{2^i} < 4 \eps.
$$
However, $P_0$ is distributed uniformly on the torus $\TT^n$. Denote $M = \EE f(P_0) = \int_{\TT^n} f$.
We have shown that $|f(P) - M| < 4 \eps$. Since $P \in \TT^n$ was an arbitrary fixed point, the lemma
follows.
\end{proof}
For $x\in\mathbb T^\infty$, denote by $F(x)$ the set of points in $\TT^{\infty}$ that coincide with $x$ in all but finitely many coordinates.

\begin{lemma}
\label{lem:Kolmogorov}
Let $A \subseteq \TT^{\infty}$ satisfy $\sigma(A) > 0$. Then $\sigma \left( \{x \in \TT^{\infty} \, ; \,  F(x)\cap A\neq\emptyset\} \right)=1$.\end{lemma}
\begin{proof}
Denote $B=\{x \in \TT^{\infty} \, ; \, F(x)\cap A\neq\emptyset\}$. Then $B$ is a measurable set which is
in fact a tail event. Since $A \subseteq B$,  Kolmogorov's zero-one law implies that $\sigma(B)=1$.\end{proof}

 For a measure space $X$ and a measurable function $f:X\to\mathbb R$, we define the essential supremum of $f$, denoted by $\mathrm{ess}\sup f$, as the supremum over all $a \in \RR$ for which the set $\{ x \in X \, ; \, f(x) > a \}$ has a non-zero measure. The definition of essential infimum is analogous.
 Define the {\it essential oscillation} of $f$ on $X$ by $$ \essosc(f;X)=\mathrm{ess}\sup f-\mathrm{ess}\inf f. $$
 Equivalently, $\essosc(f;X)=\|f(x)-f(y)\|_{L^\infty(X\times X)}$.

\begin{lemma} Let $p \in (1, \infty], 0 < \eps < 1/2$ and let $f:\mathbb{T}^{\infty}\to\mathbb{R}$ be
$1$-Lipschitz with respect to the metric $\dist_p$. Denote $q = p / (p-1)$, with $q = 1$ in case $p = \infty$. Assume that
$$
\int_{\mathbb{T}^{\infty}}\sum_{i=1}^{\infty}
c_{\eps, p, i} \cdot \Bigg|\frac{\partial f}{\partial x_{i}}\Bigg|^{q} \leq \frac{1}{2},
$$where $c_{\eps, p, i} =  2^{(i-1)^2+q (i-1)} / \left( \omega_{i-1,p} \cdot \eps^{i-1+q} \right)$.
Then
\begin{equation}
\essosc(f,\mathbb{T}^{\infty})< 8 \eps \label{eq_1156}.
\end{equation} \label{lem_548}
\end{lemma}

\begin{proof}
Let $a$ be a random point, distributed uniformly in $\TT^{\NN}$.
For a subset $S \subseteq \NN$, denote by $a_S$ the restriction of the random point $a$ to the torus
$\mathbb{T}^{S}$. Define $I_n = \{ 1, \ldots,n \}$. For $b \in \TT^{\NN \setminus I_n}$  denote
$$ \TT^n \times \{ b \} = \left \{ x \in \TT^{\infty} \, ; \, x_i = b_i \ \forall i > n \right \}. $$
For $n \geq 1$ we have
\begin{align*}
\mathbb{P}\left(  \int_{\mathbb{T}^{{n}} \times \{ a_{\mathbb{N}\setminus I_{n}} \}}\sum_{i=1}^{n}c_{\eps, p, i}\left| \frac{\partial f}{\partial x_{{i}}} \right|^{q}\geq1\right) & \leq
\mathbb{E}_{a} \int_{\mathbb{T}^{{n}}\times \{ a_{\mathbb{N}\setminus I_{n}} \}}\sum_{i=1}^{\infty}c_{\eps, p, i}
\left| \frac{\partial f}{\partial x_{{i}}} \right|^{q}\\ &
=\int_{\mathbb{T}^{\NN} } \sum_{i=1}^{\infty}c_{\eps, p, i} \left| \frac{\partial f}{\partial x_{{i}}} \right|^{q} \leq \frac{1}{2}.
\end{align*}
 Lemma \ref{Coordinate_Maurey} now implies  that for any $n \geq 1$,
\[
\mathbb{P} \left(\osc \left(f,\mathbb{T}^{{n}}\times \{ a_{\mathbb{N}\setminus I_{n}} \}\right)  <8\eps\right) \geq \frac{1}{2}.
\]
Write $B_{n}$ for the collection of all $b \in \TT^{\infty}$ for which $\osc(f,\mathbb{T}^{{n}}\times \{ b_{\mathbb{N}\setminus I_{n}} \})<8\eps$.
Obviously $B_{n+1}\subseteq B_{n}$, and by the above $\sigma (B_{n})\geq\frac{1}{2}$ for all $n \geq 1$.
Denoting $B=\cap_{n=1}^{\infty}B_{n}$, we have $\sigma (B)\geq\frac{1}{2}$. Note that
\begin{equation} B=\{ b \in \TT^\NN \, ; \, |f(x)-f(y)|<8\eps\mbox{ } \forall x,y\in F(b) \}.  \label{eq_340} \end{equation}
If (\ref{eq_1156})  does not hold, then there exist sets $C,D\subseteq \mathbb T^\NN$ of positive measure such that
for all pairs of points $c\in C$ and $d\in D$ one has $|f(c)-f(d)|\geq 8\eps$. Denote $$ \tilde C=\{x\in \mathbb T^\NN \, ; \, F(x)\cap C\neq \emptyset \} \qquad \text{and} \qquad
\tilde D=\{x\in \mathbb T^\NN \, ; \,  F(x)\cap D\neq \emptyset \}. $$ By Lemma \ref{lem:Kolmogorov}, we have $\sigma(\tilde C)=\sigma(\tilde D)=1$. Thus $\sigma(B\cap \tilde C\cap\tilde D)\geq\frac{1}{2}$, and there exist a point $b \in B$ and two elements $c\in C \cap F(b)$, $d\in D \cap F(b)$. According to the definition (\ref{eq_340}) of the set $B$,
$$ |f(c)-f(d)| < 8 \eps, $$ in contradiction.
\end{proof}

\begin{proposition} Let $p \in (1, \infty], 0 < \eps < 1/2$
and let $f:\mathbb{T}^{\infty}\to\mathbb{R}$
be a measurable function that is $1$-Lipschitz with respect to $\dist_p$.
Then there exists a parallel infinite-dimensional subtorus
$M\subseteq \mathbb{T}^{\mathbb{N}}$ such that the restriction $f|_M$ is measurable and $\essosc(f;M) \leq 8 \eps$.
\label{prop_1218} \end{proposition}

\begin{proof}
Fix a partition of  $\mathbb{N}$ into blocks $B_{1}, B_2,\ldots \subseteq \mathbb{N}$
of size
\begin{equation}
\#(B_{n})= \left \lceil \frac{2^{(n-1)^{2}+q(n-1)+(n+1)}}{\omega_{n-1} \cdot \eps^{n-1+q}} \right \rceil \qquad \qquad \qquad (n=1,2,\ldots).
\label{eq_1052} \end{equation}
In each block, choose a random element $i_{n}\in B_{n}$, independently and uniformly.
Denote $I =\{i_{1},i_2,\ldots \} \subseteq \NN$.
Additionally,
let $a$ be a random point, distributed uniformly in $\TT^{\NN}$, independent of $I$.
As before, write $q = p / (p-1)$ with $q = 1$ in case $p = \infty$.  For every fixed $n$ and for every
$b\in\mathbb{T}^{\mathbb{N}\setminus B_{n}}$, the function
$f$ restricted to $\mathbb T^{B_n}\times\{b\}$ is $1$-Lipschitz with respect to $\dist_p$. By Rademacher's theorem, for almost any $x \in \mathbb T^{B_n}\times\{b\}$ one has
$$
 \sum_{i\in B_n}  \left| \frac{\partial f}{\partial x_i} (x) \right|^q \leq 1, \label{eq_1048}
 $$ implying that
\begin{equation}
\mathbb{E}_{i_{n}}\int_{\mathbb{T}^{B_{n}}\times \{ b \}} \left| \frac{\partial f}{\partial x_{i_{n}}} \right|^{q}\leq\frac{1}{\#(B_{n})}.
\label{eq_1053} \end{equation}
Denote by $a_{\mathbb{N}\setminus B_{n}}$ the restriction of the random point $a$ to the torus
$\mathbb{T}^{\mathbb{N}\setminus B_{n}}$.
From (\ref{eq_1052}) and (\ref{eq_1053}),
\[
\mathbb{E}_{i_{n}}\int_{\mathbb{T}^{\mathbb{N}}}\frac{2^{(n-1)^{2}+q(n-1)}}{\omega_{n-1, p} \cdot \eps^{n-1+q}} \left| \frac{\partial f}{\partial x_{i_{n}}} \right|^{q} = \mathbb{E}_{i_{n},a}\int_{\mathbb{T}^{B_{n}}\times \{ a|_{\mathbb{N}\setminus B_{n}} \}}\frac{2^{(n-1)^{2}+q(n-1)}}{\omega_{n-1, p} \cdot \eps^{n-1+q}} \left| \frac{\partial f}{\partial x_{i_{n}}} \right|^{q}
\leq \frac{1}{2^{n+1}}.
\]
Denote  $ c_{\eps, p, n} = \left. 2^{(n-1)^{2}+q(n-1)} \right/  \left( \omega_{n-1, p} \cdot \eps^{n-1+q} \right)$.
Then,
\[
\mathbb{E}_{I}\int_{\mathbb{T}^{\mathbb{N}}}\sum_{n=1}^{\infty}c_{\eps, p, n}
\left| \frac{\partial f}{\partial x_{i_{n}}} \right|^{q} \leq \sum_{n=1}^{\infty} \frac{1}{2^{n+1}} = \frac{1}{2}.
\]
That is,
\begin{equation}
\mathbb{E}_{I,a_{\mathbb N\setminus I}}\int_{\mathbb{T}^{I} \times \{ a_{\NN \setminus I} \} }\sum_{n=1}^{\infty}c_{\eps, p, n} \left| \frac{\partial f}{\partial x_{i_{n}}} \right|^{q} \leq \frac{1}{2}.
\label{eq_1132}
\end{equation}
In particular, there exists a subset $I = \{ i_1,i_2,\ldots \} \subseteq \NN$
and $b \in \TT^{\NN}$ such that
\begin{equation}
\int_{\mathbb{T}^{I} \times \{ b_{\NN \setminus I} \} }\sum_{n=1}^{\infty}c_{\eps, p, n}\left| \frac{\partial f}{\partial x_{i_{n}}} \right|^{q} \leq \frac{1}{2},
\label{eq_549} \end{equation}
and such that the restriction of $f$ to the subtorus $M := \mathbb{T}^{I} \times \{ b_{\NN \setminus I} \}$ is measurable.
We may apply Lemma \ref{lem_548} thanks to (\ref{eq_549}), and conclude that $\essosc(f;M) < 8 \eps$.
\end{proof}

\begin{proof}[Proof of Theorem \ref{thm2}] Normalizing, we may assume that $f$ is $1$-Lipschitz.
Fix a sequence $\eps_n\to 0$ and apply Proposition \ref{prop_1218} in order to
construct a decreasing sequence of infinite-dimensional parallel tori $T_n$ such that $\essosc(f;T_n)<\eps_n$. Denote $a_n=\int_{T_n}{f}$.
Then  for $m>n$ one has $|a_m-a_n|<\eps_n$, implying that $a_n$ has a limit, denoted by $a$. It then follows that $a \in \RR$ satisfies the conclusion of the theorem.
\end{proof}

{
}

\end{document}